\documentclass[12pt]{amsart}

\usepackage{color,graphicx}
\usepackage[margin=1in]{geometry}
\usepackage{amsmath,latexsym,amssymb,amsfonts,amsbsy, amsthm, hyperref, mathtools, enumitem}
\usepackage{verbatim,tikz,tikz-cd}
\usepackage{bbm,bm}
\usepackage{dsfont}
\usepackage{rotating} 

\newtheorem{theorem}{Theorem}[section]
\newtheorem{lemma}[theorem]{Lemma}
\newtheorem{proposition}[theorem]{Proposition}
\newtheorem{corollary}[theorem]{Corollary}

\theoremstyle{definition}
\newtheorem{definition}[theorem]{Definition}
\newtheorem{example}[theorem]{Example}

\theoremstyle{remark}
\newtheorem{remark}[theorem]{Remark}

\def\WWW{\Omega\hskip-1.50ex\Omega}

\newcommand{\be}{\begin{equation}}
\newcommand{\ee}{\end{equation}}
\newcommand{\bee}{\begin{equation*}}
\newcommand{\eee}{\end{equation*}}
\newcommand{\ba}{\begin{aligned}}
\newcommand{\ea}{\end{aligned}}
\newcommand{\sI}{\mathcal{I}}
\newcommand{\sQ}{\mathcal{Q}}
\newcommand{\sX}{\mathcal{X}}
\newcommand{\sS}{\mathcal{S}}
\newcommand{\sL}{\mathcal{L}}
\newcommand{\sY}{\mathcal{Y}}
\numberwithin{equation}{section}

\begin{document}

\title{Shifted symplectic structure on Poisson Lie algebroid and generalized complex geometry}

\author{Yingdi Qin}
\address{Institute of Mathematics and Informatics, Bulgarian Academy of Sciences
Bulgaria, Sofia 1113, Acad. G. Bonchev Str., Bl. 8}
\email{yqin@math.bas.bg}
\thanks{Supported by the Bulgarian Ministry of Education and Science, Scientific Programme "Enhancing the Research Capacity in Mathematical Sciences (PIKOM)", No. DO1-67/05.05.2022}





\begin{abstract}
Generalized complex geometry was classically formulated by the language of differential geometry. In this paper, we reformulated a generalized complex manifold as a holomorphic symplectic differentiable formal stack in a homotopical sense. Meanwhile, by developing the machinery for shifted symplectic formal stack, we prove that the coisotropic intersection inherits shifted Poisson structure. Generalized complex branes are also studied.
\end{abstract}

\maketitle

\tableofcontents

\section{Introduction}
This paper is the first in a sequence on geometrically understanding the full complexity  of Fukaya categories of symplectic manifolds via a Lie algebroid (groupoid) approach to generalized geometry, and homotopy holomorphic structures on derived differentiable  stacks. 

Kontsevich's celebrated Homological Mirror Symmetry (HMS) conjecture asserts that the idempotent complete Fukaya Category of a symplectic manifold is derived equivalent to the category of coherent sheaves on a mirror  complex manifold. HMS provides a mathematical explanation to the mysterious Mirror Symmetry duality in quantum field theory. In the usual construction, the Fukaya category of a symplectic manifold is built in two stages. The construction starts with a geometrically constructed category with objects Lagrangian branes (Lagrangian submanifold equipped with a flat connection), morphisms given by  Lagrangian intersections, and compositions and higher associativity corrections defined by a disk instanton quantization procedure. The actual Fukaya category is then obtained as the split closure of the geometric category.

In general, the geometric category of Lagrangian branes is too small and the split closure process adds in a universal way the additional objects needed for  homological mirror symmetry to hold. Finding a geometric incarnation of these extra objects is a challenging problem that has been around for over twenty years. Analyzing the equations of motion in the requisite quantum field theories, Kapustin and Orlov\cite{KO} proposed that the missing objects should be realized as coisotropic branes - certain coisotropic submanifolds equipped with further structures. However, it is hard to define an appropriate Fukaya category including coisotropic branes since coisotropic branes behave quite differently from Lagrangian branes.

In \cite{Qc}, coisotropic branes in the Fukaya category of a symplectic torus is studied by using a doubling construction which turns a coisotropic brane into a Lagrangian brane on the doubled symplectic torus. Roughly speaking, the doubled torus of a linear symplectic torus $T$ is given by the product of $T$ with its dual torus $T^{\vee}$ with a $B$-field twisted symplectic structure. Every coisotropic brane can be lifted into a Lagrangian brane on the doubled torus.

To generalize this doubling construction to an arbitrary symplectic manifold $X$, a natural approach will be to consider a formal doubling, i.e. replace $X$ by a version of the cotangent bundle of $X$ in which the cotangent directions are purely formal. A convenient way to implement this in practice is to take the shifted cotangent bundle $T^{*}X[1]$ of the symplectic manifold $X$ as the formal replacement for the doubled torus, and to take the shifted conormal bundle $N^*S[1]$ of a coisotropic submanifold $S$ as the formal replacement for the Lagrangian lift of a coisotropic brane. This can be further applied to generalized complex geometry (which generalizes both symplectic geometry and complex geometry) and generalized complex branes (which generalize coisotropic branes on symplectic manifolds) through a Lie algebroid approach.

In this paper, we showed that a generalized complex manifold $X$ produces a homotopy holomorphic 1-shifted symplectic structure on a derived differentiable stack which is a deformation of $T^{*}X[1]$, and that a generalized complex submanifold in $X$ produces a holomorphic Lagrangian substack. As a corollary, the derived intersection of two substacks is a holomorphic symplectic substack. We may use this to study the intersection theory of two generalized complex submanifolds.  

We also prove that derived intersection of $n$-shifted coisotropic inherits a $n-1$-shifted Possion structure. This is an anologue of \cite{PTVV}'s theorem on derived Lagrangian intersection. 

The algebraic/analytic theory of shifted symplectic structures is well developed \cite{PTVV}. However, the homotopy holomorphic geometry of differentiable stacks  mentioned above is a new phenomenon in stacky  and derived geometry.  I developed a Lie algebroid and groupoid machinery which allows one to codify these novel structures, and used it to show that a homotopy holomorphic stack admits an anti-holomorphic foliation analogous  to the usual complex geometry. Hence, one can  define holomorphic sheaves on the  homotopy holomorphic stack to be the modules over this anti-holomorphic foliation.

\section{Coisotropic branes on tori}

A coisotropic brane of a symplectic manifold $(X,\,\omega)$ is a coisotropic submanifold $C$ with a $U(1)$ connection whose curvature form $F$ vanishes on the isotropic leaves in $C$ and $\omega^{-1}F$ induces a transverse complex structure on $C$. Lagrangian branes are in particular coisotropic branes whose underlying submanifolds are Lagrangian.

In \cite{Qc} I proposed a new method to extend the geometric Fukaya category of  a linear symplectic torus $(T,\,\omega)$, to include coisotropic submanifolds alongside linear Lagrangian submanifolds as objects of the extended category. The approach is by considering a twisted doubling torus $T\times T^\vee$ of $T$ and lifting  coisotropic branes into Lagrangians of $T\times T^\vee$.

\begin{example}
	Let $(T=\mathbf{R}^4/\mathbf{Z}^4,\, \omega=dr_1\wedge d\theta_1 + dr_2\wedge d\theta_2)$ be the standard symplectic four torus. 
	Then $(C=T,\,\nabla=d+2\pi ir_1d\theta_2-2\pi ir_2d\theta_1)$ is a coisotropic brane. And the induced complex stucture has complex coordinates $r_1-ir_2$, $\theta_1+i\theta_2$.
	
	The twisted double torus is  $T\times T^\vee=\mathbb{R}^4/\mathbb{Z}^4\times \mathbb{R}^4/\mathbb{Z}^4$, with the symplectic structure  $ \frac{1}{2}\omega\oplus-\frac{1}{2}\omega^{-1}=\frac{1}{2}(dr_1\wedge d\theta_1+dr_2\wedge d\theta_2+d\hat{r}_1\wedge d\hat{\theta}_1+d\hat{r}_2\wedge d\hat{\theta}_2)$. The lift of the above coisotropic brane is $(\{(r_1,\theta_1,r_2,\theta_2,\hat{r}_1,\hat{\theta}_1,\hat{r}_2,\hat{\theta}_2)\in T\times T^\vee: \hat{\theta}=r_2, \hat{\theta}_2=-r_1,\hat{r}_1=\theta_2,\hat{r}_2=-\theta \},\,\nabla=d+2\pi ir_1d\theta_2-2\pi ir_2d\theta_1)$.
\end{example}

This construction extends to general linear coisotropic branes on tori and leads to the following
\begin{proposition}[\cite{Qc}] \label{LiftisLag}\quad
	\begin{enumerate}
		\item The lift of a coisotropic linear object on $T$ is a Lagrangian submanifold of $T\times T^\vee$.
		\item The lift is a complex submanifold with respect to a canonical complex structure on $T\times T^\vee$.
	\end{enumerate}
\end{proposition}

The Lagrangian Floer theory of $T$ is naturally related to the Lagrangian Floer theory of the twisted doubling torus $T\times T^\vee$ of $T$. However, the morphism spaces that we want to consider are only certain subspaces of the Floer cohomology in doubled  torus, which we call the ``u-part".
The main result, informally, is that the Floer cohomology of two objects in $T$ is isomorphic to the ``u-part" Floer cohomology of their lifts in $T\times T^\vee$ and that this isomorphism respects the Floer products.

\begin{theorem}[Qin \cite{Qc}]
	For a pair of Lagrangian branes $L,\,L'$ which are mirror to a pair of line bundles $\mathcal{L},\,\mathcal{L'}$, let $\boldsymbol{L},\,\boldsymbol{L}'$ be the lifts in double torus. Suppose $\mathcal{L'}\otimes \mathcal{L}^{-1}$ is ample. Then the ``u-part" Floer cohomology $HF_u^*(\boldsymbol{L},\boldsymbol{L}')$ is isomorphic to $HF^*(L,\,L')$. And for two such pairs $L,L'$ and $L',L''$, the following diagram commutes
	$$\begin{tikzcd}
	{HF^*(L',\,L'')\otimes HF^*(L,\,L')} \arrow[rr] \arrow[d, "\cong"] &  & {HF^*(L,\,L'')} \arrow[d, "\cong"] \\
	HF^*_u(\boldsymbol{L}',\,\boldsymbol{L}'')\otimes HF^*_u(\boldsymbol{L},\,\boldsymbol{L}') \arrow[rr]                                 &  & HF_u(\boldsymbol{L},\boldsymbol{L}'').
	\end{tikzcd}$$
\end{theorem}
 By Theorem \ref{LiftisLag}, a coisotropic submanifold of $T$ lifts to a Lagrangian submanifold of the doubled torus which could be studied using Lagrangian Floer theory. Thus, the enlarged Fukaya category including coisotropic branes is realized as a non-full subcategory of the Fukaya category of the doubled torus. 

\section{Poisson manifold and 1-shifted symplectic stack}

\subsection{Lie algebroid and formal stack}

\begin{definition}
    A Lie algebroid over a manifold $X$ is a vector bundle $A\rightarrow X$, whose space of smooth sections is equipped with a Lie bracket which is compatible with the Lie bracket of tangent vector fields by an anchor map $a:A\rightarrow T_X$.
\end{definition}
 The importance of this concept lies in that quotient stacks by Lie algebroids (or $L_\infty$ algebroids) model (derived) formal neighborhoods of manifolds in stacks (or higher stacks). In fact, given a map $f: X\rightarrow \sX$ from a manifold $X$ to a stack (or a higher stack) $\sX$, the relative tangent complex $T_{X/\sX}$ is natually a Lie algebroid (or $L_\infty$ algebroids) over $X$. And the map $f$ natually factors as $X\rightarrow [X/T_{X/\sX}] \rightarrow \sX$, where the quotient stack $[X/T_{X/\sX}]$ is canonically identified with the (derived) formal neighborhood of $X$ in $\sX$.

 \begin{example}
     \begin{enumerate}
         \item The map $X\rightarrow pt$ factors through the De Rham stack $[X/T_X]$.
         \item A regular embedding of manifolds $Y\rightarrow X$ factors through the formal neighborhood $\Hat{X}_Y=[Y/T_{Y/X}]$ of $Y$ in $X$. 
         \item A princeple $G$-bundle $P$ on $X$ is represented by a map $X\rightarrow BG$. The relative tangent bundle is the Atiyah algebroid $At(P)$ of the princeple $G$-bundle which fits into the exact sequence
         $$0\rightarrow \mathfrak{g}\rightarrow At(P)\rightarrow T_X\rightarrow 0.$$
     \end{enumerate}
 \end{example}

\subsection{Poisson manifold and 1-shifted symplectic stack}
\begin{definition}
    Given a Poisson manifold $(X,P\in \Gamma(X, \Lambda^2 T_X))$, the \textbf{Poisson Lie algebroid} is the cotangent bundle $T^*_X \rightarrow X$ equipped with
    \begin{itemize}
        \item the anchor map 
        $P:T^*_X \rightarrow T_X$ being contraction with $P\in \Gamma(X, \Lambda^2 T_X)$. 
        \item the Lie bracket $[-,-]:\Gamma(X, T_X) \times \Gamma(X, T_X) \rightarrow \Gamma(X, T_X)$ being $[\alpha,\beta]=L_{P\alpha}\beta -\iota_{P\beta}d\alpha$.
    \end{itemize}
\end{definition}

The quotient stack $[X/T^*_X]$ by the Poisson Lie algebroid admits an $1$-shifted symplectic structure such that the quotient map $X \rightarrow [X/T^*_X]$ is Lagrangian with respect to the trivial isotropic structure.

Let $\Omega_\sX^{p,q}$ denote the space of $p$-forms of cohomological degree $q$ on a stack $\sX$.

\begin{definition}
    Given a stack $\sX$, an $n$-shifted symplectic structure on $\sX$ is 
    $$\omega=(\omega_0,\omega_1,\cdots) \in \prod_{p\geq 0} \Omega_\sX^{2+p,n-p}$$
    such that
    
\begin{itemize}
    \item $d \omega_{p-1}+\delta \omega_{p}=0$ for every $p\geq 0$ (by default $\omega_{-1}=0$).
    \item $\omega_0$ is non degenerate in the sense that it induces an equivalence 
    $\omega^\flat_0: \mathds{T}_\sX \rightarrow \mathds{T}^\vee_\sX[n]$. 
\end{itemize}
Here $d$ is the De Rham differential and $\delta$ is the internal differential.
\end{definition}

\begin{definition}[\cite{PTVV}]
    Let $(\sX,\omega)$ be a $n$-shifted symplectic stack, and a map between stacks $f:\sY \rightarrow (\sX,\omega)$.
    \begin{itemize}
        \item An isotropic structure on $f$ is a  element $\gamma \in \prod_{p\geq 0} \Omega_\sY^{2+p,n-1-p}$ such that $(d+\delta)\gamma= f^*\omega$.
        \item An Lagrangian structure on $f$ is an isotropic structure $\gamma$ such that the induced map $\gamma^\flat: \mathds{T}_{\sY/\sX}\rightarrow \mathds{T}^\vee_{\sY}[n-1]$ is an equivalence.
    \end{itemize}
    The map $\gamma^\flat: \mathds{T}_{\sY/\sX}\rightarrow \mathds{T}^\vee_{\sY}[n-1]$ is constructed as follows. Consider the following sequence of maps
    $$\mathds{T}_{\sY/\sX}
    \xrightarrow{\iota} \mathds{T}_{\sY}
    \xrightarrow{df} f^*\mathds{T}_{\sX}
    \xrightarrow{f^*\omega_0^\flat} f^*\mathds{T}^\vee_\sX[n] \xrightarrow{df^\vee} \mathds{T}^\vee_\sY[n].$$
    We can construct two null-homotopies for the composition of these maps $df^\vee \circ f^*\omega_0^\flat \circ df \circ \iota: \mathds{T}_{\sY/\sX}\rightarrow \mathds{T}_\sY^\vee[n]$. The first null-homotopy is given by the composition $df^\vee \circ\gamma_0^\flat \circ df \circ \iota :\mathds{T}_{\sY/\sX}\rightarrow \mathds{T}_\sY^\vee[n-1] $ since $\delta \gamma_0 =\omega_0$. The second null-homotopy comes from the canonical null-homotopy of $\mathds{T}_{\sY/\sX}
    \xrightarrow{\iota} \mathds{T}_{\sY}
    \xrightarrow{df} f^*\mathds{T}_{\sX}$ post-composing with $df^\vee \circ f^*\omega_0^\flat$.
    Taking the difference between the two null-homotopies, we get the map $\gamma^\flat :\mathds{T}_{\sY/\sX}\rightarrow \mathds{T}^\vee_{\sY}[n-1]$.
\end{definition}

\begin{example}\label{Poisson is Lag}
    Let $(X,P\in \Gamma(X, \Lambda^2 T_X))$ be a Poisson manifold and $\sX=[X/A]$ be the corresponding quotient stack of $X$ by its Poisson Lie algebroid $A=T^\vee_X$. The canonical closed $2$-form $(\WWW,0,0,\cdots) \in \prod_{p\geq 0} \Omega_\sX^{2+p,1-p}$ is an $1$-shifted symplectic structure on $\sX$ where 
    
    $$\WWW=d_{DR}: \Gamma(X,A) \rightarrow \Gamma(X,\wedge^2 T^\vee_X) \quad\textit{with symbol} \quad \sigma_{\WWW}=\mathds{1} : A \rightarrow T^\vee_X.$$

    The quotient map $q:X\rightarrow \sX$ admits a canonical Lagrangian structure $\gamma_q=(0,0,\cdots) \in \prod_{p\geq 0} \Omega_\sY^{2+p,-p}$.
    Indeed, the relative tangent complex of $q$ is the two term complex $T_X\oplus A \xrightarrow{(1,-P)} T_X$ sitting in degree $0,1$. The induced map
$\gamma_q^\flat:\mathds{T}_{X/\sX}\rightarrow T^\vee_X$ is the following horizontal map
    \begin{equation*}
    \begin{tikzcd}
T_X \arrow[rr]                                       &  & 0                  \\
T_X\oplus A \arrow[u,"{(1,-P)}"] \arrow[rr, "{(0,\mathbbm{1})}"] &  & T^\vee_X \arrow[u]
    \end{tikzcd}
 \end{equation*}
 which is a quasi-isomorphism of complexes.
\end{example}

\begin{theorem}[Costello-Rozenblyum, Safronov]\label{Poisson}

    Let $X$ be a differentiable manifold, then the space of $n$-shifted Poisson structures on $X$ is equivalent to the space of equivalence classes of $n+1$-shifted Lagrangian maps 
    $$X\rightarrow \mathcal{X}.$$
    $X\rightarrow \mathcal{X}$ is equivalent to $X\rightarrow \mathcal{Y}$ iff there exists another $n+1$-shifted Lagrangian map $X\rightarrow \mathcal{Z}$ and a commutative diagram 
    \begin{equation}
        \begin{tikzcd}
            & X \arrow[d] \arrow[ld] \arrow[rd]          &             \\
\mathcal{X} & \mathcal{Z} \arrow[l, "a"'] \arrow[r, "b"] & \mathcal{Y},
\end{tikzcd}
    \end{equation}
    where $a,b$ are formally étale and compatible with the Lagrangian structures. 
\end{theorem}

We only remark that given a Poisson manifold $X$, the corresponding $1$-shifted Lagrangian map can be the one in \ref{Poisson is Lag}. 

\subsection{coisotropic map and coisotropic intersection}

\begin{definition}
    Let $(X,P\in \Gamma(X, \Lambda^2 T_X))$ be a Poisson manifold. We say $f:Y \rightarrow X$ is a coisotropic map if the contraction map with the Poisson tensor $P$
    $$P: T^*X\rightarrow TX$$
    restricted to the conormal bundle $N^*_{Y/X}\rightarrow f^*T_X$ factor through $T_Y$.
\end{definition}

There is also an relative version of theorem \ref{Poisson} for coisotropic maps.

\begin{theorem}[Costello-Rozenblyum]\label{coisotropic}
    Let $Y$ and $X$ be differentiable manifolds. Then the space of coisotropic maps $f:Y\rightarrow X$ (together with Poisson structures on $X$) is equivalent to the space of equivalence classes of the commutative diagrams
    \begin{equation}
        \begin{tikzcd}
Y \arrow[r, "f"] \arrow[d, "p"] & X \arrow[d, "q"] \\
\mathcal{Y} \arrow[r, "g"]      & \mathcal{X}     
\end{tikzcd}
    \end{equation}
    where $\mathcal{X}$ is $1$-shifted symplectic, and $g,q$ are $1$-shifted Lagrangian maps such that the induced map 
    $$Y\rightarrow \mathcal{Y}\times_{\mathcal{X}}X$$
    is Lagrangian (recalling that the $1$-shifted Lagrangian intersection $\mathcal{Y}\times_{\mathcal{X}}X$ inherits a $0$-shifted symplectic structure).

\end{theorem}

\begin{theorem}[PTVV]\label{Lagintersection}
    Derived $n$-shifted Lagrangian intersection inherits a $n-1$-shifted symplectic structure.
\end{theorem}
\begin{lemma}\label{exacttriangle}
    Suppose we have 2 derived fiber products and commuting diagram
    \begin{equation}\label{2fiberproducts}
    \begin{tikzcd}
F=Y_1\times_X Y_2 \arrow[d] \arrow[r] \arrow[rrd] & Y_2 \arrow[d] \arrow[rrd] &                                                                   &                         \\
Y_1 \arrow[r] \arrow[rrd]                       & X \arrow[rrd]             & \mathcal{F}=\mathcal{Y}_1\times_\mathcal{X} \mathcal{Y}_2 \arrow[r] \arrow[d] & \mathcal{Y}_2 \arrow[d] \\
                                                &                           & \mathcal{Y}_1 \arrow[r]                                           & \mathcal{X},            
\end{tikzcd}
 \end{equation}
Then there is a distinguished triangle of complexes of sheaves on $F$ (by abuse of notation, I will not write the pullback of sheaves to $F$)
$$\mathds{T}_{X/\sX}[-1]\rightarrow \mathds{T}_{Y_1/\sY_1\times_\sX X}\oplus \mathds{T}_{Y_2/\sY_2\times_\sX X} \rightarrow \mathds{T}_{F/\mathcal{F}}.$$
\end{lemma}
We first construct the map $\mathds{T}_{X/\sX}[-1]\rightarrow \mathds{T}_{Y_1/\sY_1\times_\sX X}$. This comes from the equivalence $\mathds{T}_{X/\sX}\cong \mathds{T}_{\sY_1\times_\sX X/\sY_1}$ and the first map of the distinguished triangle
$$\mathds{T}_{\sY_1\times_\sX X/\sY_1}[-1]\rightarrow \mathds{T}_{Y_1/\sY_1\times_\sX X}\rightarrow \mathds{T}_{Y_1/\sY_1}.$$
To construct the second map $\mathds{T}_{Y_1/\sY_1\times_\sX X}\rightarrow \mathds{T}_{F/\mathcal{F}}$, note that $\mathds{T}_{Y_1/\sY_1\times_\sX X}\cong \mathds{T}_{F/\mathcal{F}\times_{\sY_2}Y_2}$, then the natural map $\mathds{T}_{F/\mathcal{F}\times_{\sY_2}Y_2}\rightarrow \mathds{T}_{F/\mathcal{F}}$ does the job.

\begin{theorem}\label{coisointerhasLag}
    If the commutative diagram \ref{2fiberproducts} santisfies additional property that the bottom square and the rightmost square
    \begin{equation*}
        \begin{tikzcd}
Y_1 \arrow[r] \arrow[d] & X \arrow[d] & Y_2 \arrow[d] \arrow[r] & X \arrow[d] \\
\mathcal{Y}_1 \arrow[r] & \mathcal{X} & \mathcal{Y}_2 \arrow[r] & \mathcal{X}
\end{tikzcd}
    \end{equation*}
    both santisfy the condition described in Theorem \ref{coisotropic}(hence $F$ is a derived coisotropic intersection), then the map
    $$F\rightarrow \mathcal{F}$$
    in \ref{2fiberproducts} is $0$-shifted Lagrangian.
\end{theorem}
\begin{proof}
    By theorem \ref{Lagintersection}, $\mathcal{F}$ inherits a $0$-shited symplectic structure. By taking the difference of the pull back of the Lagrangian structures on 
    \begin{equation}\label{LagstrY}
    Y_i\rightarrow \mathcal{Y}_i\times_{\mathcal{X}}X, \quad i=1,2,
    \end{equation}
    one get an isotropic structure $\gamma$ on
    $$F\rightarrow \mathcal{F}.$$
    We need to show this isotropic structure is Lagrangian. Indeed, the induced map $$\gamma^\flat: \mathds{T}_{F/\mathcal{F}}\rightarrow \mathds{T}_\mathcal{F}$$ fits into the commutative diagram 
    \begin{equation}
        \begin{tikzcd}
{\mathds{T}_{X/\sX}[-1]} \arrow[rr] \arrow[d] &  & \mathds{T}_{Y_1/\sY_1\times_\sX X}\oplus \mathds{T}_{Y_2/\sY_2\times_\sX X} \arrow[rr] \arrow[d] &  & \mathds{T}_{F/\mathcal{F}} \arrow[d, "\gamma^\flat"] \\
\mathds{T}^*_{X} \arrow[rr]               &  & \mathds{T}^*_{Y_1}\oplus \mathds{T}^*_{Y_2} \arrow[rr]                                           &  & \mathds{T}^*_F                                       
        \end{tikzcd}
    \end{equation}
    The middle vertical map is induced by the Lagrangian structures on \ref{LagstrY}, hence is a quasi-isomorphism. The left vertical map is induced by the $1$-shited Lagrangian structure on $X\rightarrow \mathcal{X}$, hence is a quasi-isomorphism. Since both rows are exact triangles \ref{exacttriangle}, so the right vertical map is also a quasi-isomorphism.
    
\end{proof}

\begin{corollary}
    Derived coisotropic intersection inherits a $-1$-shifted Poisson structure.
\end{corollary}
\begin{proof}
    If $F$ is a coisotropic intersection, then by theorem \ref{coisotropic} and \ref{coisointerhasLag}, there is an $0$-shited Lagrangian map $F\rightarrow \mathcal{F}$. Theorem \ref{Poisson} implies that $F$ inherits a $-1$-shifted Poisson structure.
\end{proof}
\begin{remark}
    Since Costello-Rozenblyum and Safronov's theorems hold in general for $n$-shifted case, the results here easily generalized to $n$-shifted case. This corollary is also proved by Melani-Safronov using a different method.
\end{remark}

\section{Homotopy holomorphic differentiable stacks}

In this section, I will describe an explicit polydifferential operator model for homotopy holomorphic structure on a real differentiable stack, which is the quotient of a smooth manifold by a  real Lie algebroid. From this model, one can extract a complex foliation $\mathcal{L}$ on the quotient stack, which serves as the stacky analogue of the $(0,1)$ tangent sheaf in usual complex geometry. Thus, holomorphic vector bundles in this context can be defined to be Lie algebroid modules over $\mathcal{L}$.

A Lie algebroid over a manifold $X$ is a vector bundle $A\rightarrow X$, whose space of smooth sections is equipped with a Lie bracket which is compatible with the Lie bracket of tangent vector fields by an anchor map $a:A\rightarrow T_X$. The importance of this concept lies in that Lie algebroids (and $L_\infty$ algebroids) model formal neighborhoods of manifolds in stacks (and higher stacks). In other words, the formal neighborhood of a manifold $X$ in a stack locally looks like the quotient stack $[X/A]$ of $X$ by a Lie algebroid $A$. As a first approximation, one may  think of the quotient stack $[X/A]$ as a variant of the total space of the $1$-shifted bundle $A[1] \rightarrow X$ which is twisted by the Lie algebroid structure.

If $X$ is a complex manifold and $A\rightarrow X$ is a holomorphic Lie algebroid. Then the quotient stack is naturally a holomorphic stack. However, there are geometric situations (see section \ref{GC}) where $X$ is not a complex manifold whereas the quotient stack $[X/A]$ should still be thought of as a ``holomorphic" stack. This motivates the need of developing a general theory of homotopy complex structures in derived differential geometry. The systematic study of such structures is the subject of this paper  and of the recent work \cite{PPS} of Pantev, Pym, and Safronov. The approaches taken in \cite{PPS} and this paper are different. In \cite{PPS} the authors describe  homotopy complex structures on derived stacks and their interactions with shifted symplectic geometry from first principles. To achieve this, the authors of \cite{PPS} utilize the language of derived foliations and the universal weak action of an appropriate classifying operad. In contrast,  I develop the theory in concrete local differential geometric terms utilizing the language of differential Lie algebroids and their poly differential operators.  This leads to the following definition  of homotopy holomorphic stack.

Let $\mathcal{X}$ be the formal quotient stack for the action of a real Lie algebroid $A$ on a manifold $X$. Let $\Omega(T_{\mathcal{X}})$ be the space of differential forms with values in the tangent complex of $\mathcal{X}$. By its nature $\Omega(T_{\mathcal{X}})$ is a bi-complex where the bi-grading of the terms is labeled by the form degree $p$ and cohomological degree $d$. An element in $\Omega^{p}(T_{\mathcal{X}})^{d}$ has total degree $p+d$.

\begin{definition}
$\mathcal{X} = [X/A]$ be the formal quotient of a manifold by a real Lie algebroid.
A {\em\bfseries homotopy holomorphic structure} on $\mathcal{X}$ is a pair $(\mathcal{I},\mathcal{Q})$ where
\bee
\mathcal{I}=\sI_1+\sI_2+\cdots\in \prod_{p\geq 1} \Omega^{p}(T_{\mathcal{X}})^{1-p},
\;
\mathcal{Q}=\sQ_1+\sQ_2+\cdots\in \prod_{p\geq 1} \Omega^{p}(T_{\mathcal{X}})^{-p}
\eee
such that 

\begin{alignat}{3}
\label{IQcom}
[\sI,\sQ]_{NR} & =0  \\
\label{IMC}
\delta \mathcal{I} +\frac{1}{2}[\sI,\sI]_{FN} & =0 \\
\label{Isq}
\delta_{\sI}\sQ+\frac{1}{2}[\sI,\sI]_{NR} & =-1
\end{alignat}

\end{definition}

\begin{remark} Here $[-,-]_{FN}$ and $[-,-]_{NR}$ denote the Fr\"{o}licher-Nijenhuis and the Nijenhuis-Richardson brackets respectively.  In the given data $\sI_1\in End(T_{\sX})$ is the main term of the homotopy holomorphic structure, and all other terms are homotopy data similar to the key describing the closedness of a form in derived geometry as in \cite{PTVV}. The condition \eqref{Isq} implies that $\sI_1$ squares to $-1$ up to a homotopy given by $\sQ_1$. The condition \eqref{IMC} implies that $\sI_1$ is a multiplicative $(1,1)$ tensor whose Nijenhuis torsion vanishes up to a homotopy given by $\sI_2$. Finally, the condition \eqref{IQcom} implies that $\sI_1$ commutes with $\sQ_1$.
\end{remark}

We would like to understand  holomorphic bundles on a homotopy holomorphic stack. Unfortunately, we don't have the structure sheaf of holomorphic functions, neither do we have a $\bar{\partial}$ operator. To define a holomorphic vector bundle on $\sX$, we will use an anti-holomorphic foliation, which is a counterpart of the $T^{0,1}$ sheaf  in usual complex geometry. 

\begin{definition}
An {\em\bfseries anti-holomorphic foliation} on $\sX$ is a complex Lie algebroid $\sL$ over $\sX$ with anchor map $\rho:\sL \rightarrow T_{\sX}^\mathbb{C}$ such that 
\be
\rho\oplus \bar{\rho} : \sL \oplus \bar{\sL} \rightarrow T_{\sX}^\mathbb{C}
\ee
is a quasi-isomorphism.
\end{definition}

\begin{remark} The anti-holomorphic foliation is only special case of the more general complex derived foliations on $\mathcal{X}$ one might work with, but it is enough for our purposes.
\end{remark}

\begin{theorem}\label{holofoli} 
Let $(\sX,\sI,\sQ)$ be a homotopy holomorphic stack where $\sX$ is the quotient stack of a manifold by the action of real Lie algebroid. Then $\sX$ admits a canonical anti-holomorphic foliation $\sL$ together with a homotopy $\gamma:\sL \rightarrow T_{\sX}[-1]$ such that the following diagram commutes
\be
\begin{tikzcd}
\mathcal{L} \arrow[d, "-i"'] \arrow[r, "\rho"] & T_{\mathcal{X}} \arrow[d, "\mathcal{I}_1"] \arrow[ld, "\gamma"', Rightarrow, shift left] \\
\mathcal{L} \arrow[r, "\rho"']                 & T_{\mathcal{X}}                                                                         
\end{tikzcd}
\ee
\end{theorem}

This theorem generalizes the result on holomorphic Lie algebroid in \cite{BD}.

Now we are ready to define holomorphic vector bundles over the holomorphic stack. Recall that a holomorphic vector bundle over a complex manifold is equivalent to a Lie algebroid module over $T^{0,1}$.

\begin{definition}
A holomorphic sheaf over $\sX$ is a weak Lie algebroid module over $\sL$.
\end{definition}

\section{Generalized complex manifolds as holomorphic symplectic stacks}\label{GC}

Generalized complex geometry was introduced by Hitchin and was further developed by Gualtieri and Cavalcanti. A generalized complex manifold is a smooth manifold $X$ equipped with a bundle map $\mathbf{J}\in End(TX\oplus T^*X)$ such that $\mathbf{J}^2=-1$, $\mathbf{J}$ preserves the natural pairing of $TX\oplus T^*X$ and $\mathbf{J}$ is integrable under the Courant bracket.

If we write 
$\mathbf{J}=\left(
\ba
-&I  & P    \\
&Q    & ^t\!I
\ea
\right)$,
then $P$ is a Poisson bi-vector by the integrability condition. So we can form a Poisson Lie algebroid $T^*_X \rightarrow X$ with anchor map given by the contraction with $P$. Denote the corresponding quotient stack $[X/T^*_X]$ by $\sX$.

This formal stack $\sX$ can be thought of as a deformation of the total space of the $1$-shifted cotangent bundle. The tangent complex of $\sX$ (pulled back to the atlas $X$) is a two term complex $T^*_X \rightarrow T_X$ concentrated in degrees $-1,0$, and the cotangent complex of $\sX$ is  $T^*_X \rightarrow T_X$ which is concentrated in degrees $0,1$. The tautological identification of these two complexes gives rise to a $1$-shifted symplectic structure.

\begin{proposition}
$\sX$ is naturally equipped with a $1$-shifted symplectic structure. The natural map $X\rightarrow \sX$ is Lagrangian with the Lagrangian structure corresponding to the intrinsic quasi-isomorphism between the relative tangent complex of $X \to \sX$ and the cotangent bundle of $X$.
\end{proposition}

In the following, we investigate the geometric structures on $\sX$ induced by the generalized complex structure on $X$.

Alan Weinstein\cite{W} and various other authors suggested that a generalized complex manifold should give rise to a "holomorphic" symplectic groupoid integrating the above Poisson Lie algebroid. In our context, the symplectic groupoid envisioned by A. Weinstein should be a presentation of a geometric $1$-shifted symplectic stack which integrates the canonical $1$-shifted symplectic structure on the formal stack  $\sX$ described in the previous proposition. Using the poly-differential operator formalism of homotopy complex structures that I developed, I was able to confirm a conjecture of Pantev, Pym and Safronov predicting that a  generalized complex structure on $X$ induces a homotopy holomorphic structure on $\sX$ so that the $1$-shifted symplectic structure on $\sX$ becomes holomorphic with respect to this homotopy holomorphic structure. This provides a pleasant explanation of A. Weinstein's putative  ``holomorphic" symplectic structure directly at the level of the Poisson Lie algebroid. The precise  result is summarized in the following theorem. A version of this theorem  is also proven in \cite{PPS} with a different argument which relies on the language of derived foliations. The theorem can also be viewed as a recasting in the language of formal stacks of the Bailey-Gualtiery local structure theorem for generalized complex manifolds \cite{BG}.

\begin{theorem} 
The generalized complex structure on $X$ induces a homotopy holomorphic structure $(\sI,\sQ)$ on $\sX$. Together with the $1$-shifted symplectic structure, $\sX$ becomes a homotopy holomorphic $1$-shifted symplectic stack.
Conversely, if the quotient stack $\sX$ of a real Lie algebroid $A\rightarrow X$ is equipped with a homotopy holomorphic $1$-shifted symplectic structure such that $X\rightarrow \sX$ is Lagrangian, then $A$ is isomorphic to a Poisson Lie algebroid over $X$ and the holomorphic $1$-shifted symplectic structure is induced by a unique generalized complex structure on $X$.
\end{theorem}

\

\begin{remark}
Roughly speaking, $\sI_1$ of the homotopy holomorphic structure is given by the diagonal term $I,\,^t\!\!I$
of the generalized complex structure
$\mathbf{J}=\left(
\ba
-&I  & P    \\
&Q    & ^t\!I
\ea
\right)$.
$\sQ_1$ is given by $Q$. $\sI_2$ is related to $Q,I,\,^t\!I$ and the Courant twisting class. In this case, $\sI=\sI_1+\sI_2$, $\sQ=\sQ_1$ and higher terms vanish for degree reason.
\end{remark}

\

\begin{remark}
To define a homotopy holomorphic $1$-shifted symplectic structure, one needs extra homotopy data and compatibility conditions between the homotopy holomorphic and $1$-shifted symplectic structures.
\end{remark}

\begin{remark}
This theorem generalizes a result in \cite{SX} at the level of the Poisson Lie algebroid, which states that the generalized complex structure on $X$ induces an symplectic quasi-Nijenhuis structure on the associated symplectic groupoid.
\end{remark}

The above theorem identifies generalized complex manifolds with homotopy holomorphic $1$-shifted symplectic formal stacks. It provides another natural view point on generalized geometry.

In the what follows, we will consider generalized complex submanifolds (and branes) using this view point. According to Gualtieri \cite{Gu}, a  generalized complex submanifold $(S,F)$ of a generalized complex manifold $X$ is a coisotropic submanifold $S$ with a Courant trivialization $F$, such that the generalized tangent bundle $\tau_{S}^{F}$ is stable under the generalized complex structure $\mathbf{J}$. As a consequence, $\tau_{S}^{F}\otimes \mathbb{C}=L_{S}\oplus \bar{L}_S$ and $L_{S}$ is a complex Lie algebroid. A generalized complex brane is a generalized submanifold equipped with a Lie algebroid module over $L_{S}$ \cite{Gu}.

\begin{lemma}
If $S$ is a coisotropic submanifold of a Poisson manifold $(X,P)$, then the conormal bundle $N^*_{S}$ naturally inherits a Lie algebroid structure. (Recall that the quotient stack $\sX=[X/T^*_X]$ has a canonical 1-shifted symplectic structure.) And the natural map between the quotient stacks $\sS=[S/N^*_{S}] \rightarrow \sX$ is Lagrangian.
\end{lemma}

The next theorem identifies a generalized complex submanifold $S$ with a holomorphic Lagrangian substack of $\sX$.
Again a version of this theorem with a different more abstract proof can be found in \cite{PPS}. 

\begin{theorem}
The generalized complex submanifold structure on $S$ induces a homotopy holomorphic structure on $\sS$ and an isotropic structure such that $\sS \rightarrow \sX$ is a holomorphic Lagrangian map. 
Conversely, If $\sS$ is the quotient stack of a Lie algebroid, and $\sS \rightarrow \sX$ is a holomorphic Lagrangian map. Then $\sS$ is induced from a generalized complex submanifold.
\end{theorem}

By theorem \ref{holofoli}, the holomorphic substack $\sS$ has an anti-holomorphic foliation $\sL_S$.

\begin{corollary}
A generalized complex brane $(S,F,V)$ induces a holomorphic substack $\sS \rightarrow \sX$ equipped with a holomorphic sheaf $\mathcal{V}$ on $\sS$.
\end{corollary}

In conclusion, we turned generalized complex manifold into holomorphic 1-shifted symplectic stack, and turned generalized complex branes into holomorphic Lagrangian substacks equipped with holomorphic sheaves. This framework provides a new view point on the potential construction of the ``Fukaya category" of a generalized complex manifold as an analogue of derived category of coherent sheaves on a holomorphic stack.

\begin{example}
Let $(X,\omega)$ be a symplectic manifold. Then the associated generalized complex structure is $\mathbf{J_\omega}=\left(
\ba
&0  & \omega^{-1}    \\
&-\omega    & 0
\ea
\right)$. $\omega^{-1}$ is a Poisson bivector and induces the Poisson Lie algebroid structure  $A\rightarrow X$. Let $\sX=[X/A]$ denote the formal quotient of $X$ by its Poisson Lie algebroid. The homotopy holomorphic structure induced by $\mathbf{J_\omega}$ specilizes to $\sI=0$, $\sQ=\sQ_1=-\omega$. And the anti-holomorphic  foliation \ref{holofoli} on $\sX$ is the descent of the complexified tangent algebroid of $X$. 
A Lagrangian submanifold $S$ of $X$ induces the substack $\sS$ of the formal quotient of $S$ by its tangent algebroid. The homotopy holomorphic structure on $\sS$ is $\sI=0$, $\sQ=\sQ_1=\mathbf{1}_{T_S}$. The anti-holomorphic foliation on $\sS$ is again the descent of the complexified tangent algebroid of $S$. 
A space filling generalized complex submanifold $C=X$ of $X$ induces the substack $\mathcal{C}$ of the formal quotient of $C$ by the trivial Lie algebroid. So $\mathcal{C}$ is indeed $C$ itself. The homotopy holomorphic structure on $\mathcal{C}$ is an ordinary holomorphic structure on $C$ induced by the generalized complex submanifold structure. 
\end{example}

\section{Jet bundle and Atiyah algebroid}

Let $E\rightarrow X$ be a vector bundle, $J^1E={pr_1}_*(\mathcal{O}_{X\times X}^{(1)}\otimes pr_2^*E)$ be its $1$-jet bundle. A local section of $J^1E$ can be represented by $g(x,y)\in \Gamma(X\times X,pr_2^*E)$. The local section of $J^1E$ represented by $g(x,y)$ should be thought of as the first order Taylor expansion in $y$ variable centered at $x$, i.e. $g(x,x)+D_yg(x,x)$.  We have a short exact sequence
\begin{equation}
    0\rightarrow T^\vee \otimes E \rightarrow J^1E \rightarrow E\rightarrow 0
\end{equation}
The first map above is given by $df\otimes s \mapsto g(x,y)=(f(y)-f(x))s(x)$.
If we apply $Hom(-,E)$ to the above SES, we get
\begin{equation}
    0\rightarrow Hom(E,E)\rightarrow Diff^1(E,E)\xrightarrow{symbol} T\otimes Hom(E,E)\rightarrow 0
\end{equation}
Here $Diff^1(E,E)=Hom(J^1E,E)$ is the bundle of first order differential operators on $E$ , whose symbol lies in $T\otimes Hom(E,E)$.
Atiyah algebroid $At(E)$ is the bundle of differential operators on $E$ with symbol land in $T\otimes Id$. $At(E)$ is equipped with a Lie bracket given by the commutator of differential operators. Then we Have the SES for Atiyah algebroid
\begin{equation}
    0\rightarrow Hom(E,E)\rightarrow At(E)\xrightarrow{symbol} T\rightarrow 0.
\end{equation}
A more general approach to introduce the Atiyah algebroid is through $G$-principle bundle. Let $P\rightarrow X$ be a $G$-principle bundle with right $G$ action.

\section{Further remark}

\begin{enumerate}
    \item The definition of a homotopy holomorphic structure readily generalizes to the quotient stack of a real $L_\infty$ algebroid and possibly more general real derived stacks. I would expect a result like theorem \ref{holofoli} to hold in the general situation. With such a theorem, one can define the holomorphic sheaves on homotopy holomorphic derived stacks. 
    
    \item Generalized complex geometry is also called generalized geometry of type $D_n$. Besides $D_n$ geometry, Hitchin's school has also developed generalized geometries of type $B_n, C_n$. My expectation is that $B_n$ geometry can be translated to homotopy holomorphic 1-shifted contact stack, and $C_n$ geometry to homotopy holomorphic 1-shifted Riemannian stacks. In addition, there is a T-duality expected between $B_n$ and $C_n$ geometry, which is closely  related to Langlands duality between $B_n$ and $C_n$ type Lie groups.
    
    \item One motivation for generalized complex geometry is the mirror symmetry between the symplectic and complex geometry. It is expected that mirror symmetry could be treated in the greater context of generalized complex geometry. People want to define a ``Fukaya category" for a generalized complex geometry which unifies the Fukaya category of a symplectic manifold and the derived category of coherent sheaves of a complex manifold. The framework of homotopy holomorphic 1-shifted symplectic stack potentially leads to such a definition. 
    
    \item In usual complex geometry, we have the celebrated Newlander-Nirenberg theorem which asserts that a holomorphic manifold admits an atlas of holomorphic charts. An abstract analogue of this theorem using universal action of an appropriate operad is proven in \cite{PPS}. It would be useful if our explicit poly-differential operator formalism can be refined to produce holomorphic charts of a homotopy holomorphic stack by establishing a theorem analogous to Newlander-Nirenberg theorem. For the stacks arising from generalized complex manifolds, this is a consequence of the powerful local structure theorem for generalized complex manifolds of Bailey and Gualtieri \cite{BG}. For general derived stacks, the explicit version of the Newlander-Nirenberg theorem will likely require a Spencer cohomology analysis of higher obstructions to integrablity. 
    
\end{enumerate}

\bibliographystyle{amsplain}

\end{document}